\documentclass{amsart}

\usepackage{amsmath}
\usepackage{amsthm}
\usepackage{amsfonts}
\usepackage{amssymb}

\newcommand{\E}{\mathbb{E}}
\newcommand{\Prob}{\mathbb{P}}
\newcommand{\R}{\mathbb{R}}

\newcommand{\N}{\mathbb{N}}

\newcommand{\ud}{\mathrm{d}}
\newcommand{\udx}{\mathrm{dx}}
\newcommand{\udy}{\mathrm{dy}}

\newcommand{\var}{\mathrm{Var}}
\newcommand{\cov}{\mathrm{Cov}}
\newcommand{\tr}{\mathrm{Tr}}
\renewcommand\Re{\operatorname{Re}}
\renewcommand\Im{\operatorname{Im}}

\newcommand{\GUE}{\mathrm{GUE}}
\newcommand{\GOE}{\mathrm{GOE}}
\newcommand{\GSE}{\mathrm{GSE}}

\theoremstyle{plain}
  \newtheorem{theorem}{Theorem}

  \newtheorem{lemma}[theorem]{Lemma}
  \newtheorem{corollary}[theorem]{Corollary}

\theoremstyle{definition}
  \newtheorem{definition}[theorem]{Definition}
  
  \newtheorem{remark}[theorem]{Remark}

\begin{document}
\title{Gaussian Fluctuations of Eigenvalues in Wigner Random Matrices}
\author[S. O'Rourke]{Sean O'Rourke}
\address{Department of Mathematics, University of California, Davis, One Shields Avenue, Davis, CA 95616-8633  }
\email{sdorourk@math.ucdavis.edu}

\begin{abstract}
We study the fluctuations of eigenvalues from a class of Wigner random matrices that generalize the Gaussian orthogonal ensemble.  

We begin by considering an $n \times n$ matrix from the Gaussian orthogonal ensemble (GOE) or Gaussian symplectic ensemble (GSE) and let $x_k$ denote eigenvalue number $k$.  Under the condition that both $k$ and $n-k$ tend to infinity as $n \rightarrow \infty$, we show that $x_k$ is normally distributed in the limit.  

We also consider the joint limit distribution of eigenvalues $(x_{k_1},\ldots,x_{k_m})$ from the GOE or GSE where $k_1$, $n-k_m$ and $k_{i+1} - k_i$, $1 \leq i \leq m-1$, tend to infinity with $n$.  The result in each case is an $m$-dimensional normal distribution.  

Using a recent universality result by Tao and Vu, we extend our results to a class of Wigner real symmetric matrices with non-Gaussian entries that have an exponentially decaying distribution and whose first four moments match the Gaussian moments. 
\end{abstract}

\maketitle

\section{Introduction and Formulation of Results}

In this paper, we study the classical ensemble of random matrices introduced by Eugene Wigner in the 1950s, \cite{wig}.  In particular, we will consider Wigner real symmetric matrices and Wigner Hermitian matrices.  We begin with the real symmetric case.  

\subsection{Real Symmetric Wigner Matrices}

Following Tao and Vu in \cite{tao}, we define a class of Wigner real symmetric matrices with exponential decay.

\begin{definition}[Wigner real symmetric matrices]
Let $n$ be a large number.  A Wigner real symmetric matrix (of size $n$) is defined as a random real symmetric $n \times n$ matrix $M_n = \left( m_{ij} \right)_{1\leq i,j \leq n}$ where
\begin{itemize}
	\item For $1 \leq i < j \leq n$, $m_{ij}$ are i.i.d. real random variables.
	\item For $ 1 \leq i \leq n$, $m_{ii}$ are i.i.d. real random variables.
	\item The entries $m_{ij}$ have exponential decay i.e. there exists constants $C,C'>0$ such that $\Prob\left(| m_{ij} | \geq t^C \right) \leq \exp(-t)$, for all $t \geq C'$.
\end{itemize}
\end{definition}

The prototypical example of a Wigner real symmetric matrix is the Gaussian orthogonal ensemble (GOE).  The GOE is defined by the probability distribution on the space of $n \times n$ real symmetric matrices given by 
\begin{equation} \label{prob_dist_GOE}
	\Prob(\ud H) = C^{(\beta)}_n e^{-\frac{\beta}{2}\tr{H^2}} \ud H
\end{equation}
where $\beta = 1$ and $\ud H$ refers to the Lebesgue measure on the $\frac{n(n+1)}{2}$ different elements of the matrix.  So for a matrix $H=\left( h_{ij} \right)_{1 \leq i,j \leq n}$ drawn from the GOE, the elements
\begin{equation*}
	\{ h_{ij}; 1 \leq i \leq j \leq n \}
\end{equation*}
are independent Gaussian random variables with zero mean and variance $\frac{1+\delta_{ij}}{2}$.  

\subsection{Wigner Hermitian Matrices}
Similar to the real symmetric case, we define Wigner Hermitian matrices.  

\begin{definition}[Wigner Hermitian matrices]
Let $n$ be a large number.  A Wigner Hermitian matrix (of size $n$) is defined as a Hermitian $n \times n$ matrix $M_n = \left( m_{ij} \right)_{1\leq i,j \leq n}$ where
\begin{itemize}
	\item $\{ \Re m_{ij}, \Im m_{ij} : 1 \leq i < j \leq n \}$ are a collection of i.i.d. real random variables.
	\item For $ 1 \leq i \leq n$, $m_{ii}$ are i.i.d. real random variables.
	\item The entries $m_{ij}$ have exponential decay i.e. there exists constants $C,C'>0$ such that $\Prob\left(| m_{ij} | \geq t^C \right) \leq \exp(-t)$, for all $t \geq C'$.
\end{itemize}
\end{definition}

The classical example of a Wigner Hermitian matrix is the Gaussian unitary ensemble (GUE).  The GUE is defined by the probability distribution given in \eqref{prob_dist_GOE} with $\beta = 2$, but on the space of $n \times n$ Hermitian matrices.  Thus for a matrix $H=\left( h_{ij} \right)_{1 \leq i,j \leq n}$ drawn from the GUE, the $n^2$ different elements of the matrix, 
\begin{equation*}
	\{\Re h_{ij}; 1 \leq i \leq j \leq n, \Im  h_{ij}; 1 \leq i < j \leq n \}
\end{equation*}
are independent Gaussian random variables with zero mean and variance $\frac{1+\delta_{ij}}{4}$.

\subsection{Gaussian Symplectic Ensemble}
Historically, quaternion self-dual Hermitian Wigner matrices have not been studied.  We will, however, introduce the Gaussian symplectic ensemble (GSE).  The GSE is defined by the probability density given in \eqref{prob_dist_GOE} with $\beta = 4$, but on the space of $n \times n$ quaternion self-dual Hermitian matrices.  For a matrix $H=\left( h_{ij} \right)_{1 \leq i,j \leq n}$ drawn from the GSE, there are $n(2n-1)$ distinct real members of the matrix,
\begin{equation*}
	\{ h^{(0)}_{jk}; 1 \leq j \leq k \leq n, h^{(i)}_{jk}; 1 \leq j < k \leq n \text{ for }i=1,2,3 \}
\end{equation*}
where each quaternion entry is given by 
\begin{equation*}
	h_{jk} = h^{(0)}_{jk} + h^{(1)}_{jk}e_1 + h^{(2)}_{jk} e_2 + h^{(3)}_{jk}e_3.
\end{equation*}
Here $\{1, e_1, e_2, e_3\}$ denotes the standard quaternion basis with the usual multiplication table,
\begin{align*}
	&e_1^2 = e_2^2 = e_3^2 = -1 \\
	e_1 e_2 = -e_2 e_1 = e_3 \qquad &e_1 e_3 = -e_3 e_1 = e_2 \qquad e_1 e_2 = -e_2 e_1 = e_3 .
\end{align*}

The entries are again independent Gaussian random variables with zero mean.  For $j<k$, $h^{(i)}_{jk}$ has variance $\frac{1}{8}$ for each $i=0,1,2,3$ and $h^{(0)}_{jj}$ has variance $\frac{1}{4}$.  

\subsection{Distribution of Eigenvalues for the Gaussian Ensembles}
In each of the Gaussian ensembles above, there is an induced measure of the corresponding $n$ real eigenvalues $x_i$.  The induced measure can be calculated (see Mehta's book, \cite{me}) and it's density is given by 
\begin{equation*}
	p_n^{(\beta)}(x_1,\ldots,x_n) = \frac{1}{Z_n^{(\beta)}} \prod_{1 \leq i < j \leq n} \left| x_i - x_j \right|^\beta e^{-\frac{\beta}{2}(x_1^2 + \cdots + x_n^2)}
\end{equation*}
where $\beta = 1,2,$ or $4$ corresponds to the GOE, GUE, or GSE, respectively and $Z_n^{(\beta)}$ is a normalizing constant.  

Since the spectrum is simple with probability $1$, we can further order the eigenvalues so that $x_1 < x_2 < \ldots < x_n$.  This ordering gives the probability density $\rho^{(\beta)}_{n,n}(x_1,\ldots,x_n)$ of the ordered eigenvalues on the space
\begin{equation*}
	\R^n_{\mathrm{ord}} = \{(x_1,\ldots,x_n) : x_1 < \ldots < x_n \}.
\end{equation*}
Here
\begin{equation*}
	\rho_{n,n}^{(\beta)}(x_1,\ldots,x_n) = n! p_n^{(\beta)}(x_1,\ldots,x_n).
\end{equation*}

We can define the correlation functions for the eigenvalues as
\begin{equation*}
	\rho^{(\beta)}_{n,k}(x_1,\ldots,x_k) = \frac{n!}{(n-k)!} \int_{\R^{n-k}} p_n^{(\beta)}(x_1,\ldots,x_n)\ud x_{k+1}\ldots \ud x_n.
\end{equation*}

In the case of the GUE, the eigenvalues form a determinantal random point process.  In this case, 
\begin{equation*}
	\rho^{(2)}_{n,k}(x_1,\ldots,x_k) = \det (K_n(x_i,x_j))_{i,j=1}^k.
\end{equation*}
Where the kernel $K_n(x,y)$ is given by
\begin{equation*}
	K_n(x,y) = \sum_{i=0}^{n-1} \phi_i(x) \phi_i(y) e^{-\frac{1}{2}(x^2 + y^2)}
\end{equation*}
and $\phi_i$ are the orthonormal Hermite polynomials, i.e. $\int_{\R} \phi_i(x) \phi_j(x) e^{-x^2}\udx = \delta_{ij}$.  All these results and more can be found in Mehta's book, \cite{me} as well as Deift's books, \cite{de} and \cite{de2}.  

\subsection{Wigner's Semicircle Law}
An important result regarding Wigner random matrices is the famous semicircle law.  Denote by $\rho_{\sigma}$ the semicircle density function with support on $[-2\sigma, 2\sigma]$,
\begin{equation*}
	\rho_{\sigma} (x) = 
	\begin{cases} 
		\frac{1}{2\pi\sigma^2} \sqrt {4\sigma^2-x^2}, &|x| \leq 2\sigma, \\ 
		0, &|x| > 2\sigma. 
	\end{cases} 
\end{equation*}

\begin{theorem}[Semicircle Law] \label{semicircle}
Let $M_n = \left(m_{ij}\right)_{1 \leq i, j \leq n}$ be a Hermitian Wigner matrix where $m_{ij}$ has variance $\sigma^2$ for $1 \leq i < j \leq n$.  If $x_1 \leq x_2 \leq \ldots \leq x_n$ denote the ordered eigenvalues of $\frac{1}{\sqrt{n}} M_n$, then as $n \rightarrow \infty$, 
\begin{equation*}
	\frac{1}{n} \# \left\{ 1 \leq i \leq n : x_i \leq x \right\} \longrightarrow \int_{-2\sigma}^{x} \rho_{\sigma}(y) \udy
\end{equation*}
almost surely where $\#\{.\}$ denotes the number of elements in the set indicated.  
\end{theorem}

A similar result holds as well for real symmetric Wigner matrices.  For a discussion of such results as well as a proof of Theorem \ref{semicircle} see \cite{bai}, \cite{pa}, and \cite{wig}.  

\subsection{Main Results}
In \cite{gu}, Gustavsson studies the distribution of eigenvalue number $k$, $x_k$, of the GUE when both $k$ and $n-k$ tend to infinity as $n \rightarrow \infty$.  For example, if $k = n - \log n$, then for large values of $n$, $x_k$ is relatively close to the right edge of the spectrum.  As another example, consider when $k=n/2$.  In this case, $x_k$ is in the middle of the spectrum.  In each case, Gustavsson showed that $x_k$ is normally distributed in the limit (see Theorems \ref{thm_bulk} and \ref{thm_edge} below when $\beta=2$ and $m=1$).

Gustavsson also considers the joint distribution of several eigenvalues $(x_{k_1},\ldots,x_{k_m})$ from the GUE where $k_1$, $n-k_m$ and $k_{i+1} - k_i$, $1 \leq i \leq m-1$, tend to infinity with $n$.  In this case, Gustavsson showed that the limiting distribution is an $m$-dimensional normal distribution (see Theorems \ref{thm_bulk} and \ref{thm_edge} below when $\beta=2$ and $m>1$).

In recent months, there have been a number of universality results for Wigner matrices.  In particular, it was shown that in the limit as $n \rightarrow \infty$, the statistical properties of $m$ eigenvalues from a Wigner matrix with exponential decay are independent of the probability distribution of the matrix entries (Mehta discusses the universality conjecture in his book, \cite{me}, see Conjectures 1.2.1 and 1.2.2).  For further details see results by Tao and Vu in \cite{tao} and \cite{tao2}, results by Erd\H{o}s, Schlein, and Yau, in \cite{yau}, \cite{yau2}, \cite{yau3}, and \cite{yau4}, and combined results by Erd\H{o}s, Ramirez, Schlein, Tao, Vu, and Yau in \cite{yau5}.  

Among their many results in \cite{tao} and \cite{tao2}, Tao and Vu prove that in the limit as $n \rightarrow \infty$, the fine spacing statistics for a Wigner random matrix are only determined by the first four moments of the entries.  As a consequence, Tao and Vu extend Gustavsson's results for the GUE to a class of Hermitian Wigner matrices with non-Gaussian entries whose first four moments match the Gaussian moments (see Corollary \ref{hermitian_wigner} below).

In this paper, we extend Gustavsson's results for the GUE to the GOE and GSE.  Then the powerful machinery developed by Tao and Vu generalizes our results to a class of real symmetric Wigner matrices with non-Gaussian entries.

\begin{remark}
In \cite{tw}, Tracy and Widom studied the distribution of the smallest and largest eigenvalues in the GUE.  Later in \cite{tw2}, they extended the result to include the smallest and largest eigenvalues in the GOE and GSE.
\end{remark}

For the theorems below we define
\begin{equation*}
	G(t) = \frac{2}{\pi} \int_{-1}^t \sqrt{1-x^2}\udx \qquad -1 \leq t \leq 1 .
\end{equation*}

Following Gustavsson's notation, we write $k(n) \sim n^\theta$ to mean that $k(n) = h(n) n^\theta$ where $h$ is a function such that, for all $\epsilon > 0$,
\begin{equation*}
	\frac{h(n)}{n^\epsilon} \longrightarrow 0 \text{ and } h(n) n^\epsilon \longrightarrow \infty
\end{equation*}
as $n \rightarrow \infty$.

\subsubsection{Results for the Gaussian Ensembles}
Below we present the main theorems which extend Gustavsson's results to the GOE and GSE.  

\begin{theorem}[The bulk] \label{thm_bulk}
Let $x_1 < x_2 < \cdots < x_n$ be the ordered eigenvalues from a random matrix drawn from the GOE, GUE, or GSE.  Consider $\{ x_{k_i} \}_{i=1}^m$ such that $0 < k_i - k_{i+1} \sim n^{\theta_i}$, $0 < \theta_i \leq 1$, and $\frac{k_i}{n} \rightarrow a_i \in (0,1)$ as $n \rightarrow \infty$.  Define $s_i = s_i(k_i,n) = G^{-1}(k_i/n)$ and set
\begin{equation*}
	X_i = \frac{ x_{k_i} - s_i \sqrt{2n} } { \left( \frac{ \log n} {2\beta(1-s_i^2)n} \right)^{1/2} } \quad i=1,\ldots,m
\end{equation*}
where $\beta=1,2,4$ corresponds to the GOE, GUE, or GSE.  Then as $n \rightarrow \infty$,
\begin{equation*}
	\Prob[X_1 \leq \xi_1,\ldots,X_m \leq \xi_m] \longrightarrow \Phi_{\Lambda}(\xi_1,\ldots,\xi_m)
\end{equation*}
where $\Phi_\Lambda$ is the cdf \footnote{Cumulative distribution function} for the $m$-dimensional normal distribution with covariance matrix $\Lambda_{i,j} = 1-\max\{\theta_k : i \leq k < j < m \}$ if $i < j$ and $\Lambda_{i,i} = 1$.
\end{theorem}

\begin{theorem}[The edge] \label{thm_edge}
Let $x_1 < x_2 < \cdots < x_n$ be the ordered eigenvalues from a random matrix drawn from the GOE, GUE, or GSE.  Consider $\{x_{n-k_i}\}_{i=1}^m$ such that $k_1 \sim n^{\gamma}$ where $0 < \gamma < 1$ and $0 < k_{i+1} - k_i \sim n^{\theta_i}$, $0<\theta_i < \gamma$.  Set
\begin{equation*}
	X_i = \frac{ x_{n-k_i} - \sqrt{2n} \left( 1 - \left( \frac{3\pi k_i}{4 \sqrt{2} n} \right)^{2/3} \right) } { \left( \left( \frac{1}{12\pi} \right)^{2/3} \frac{ 2\log k_i }{\beta n^{1/3} k_i^{2/3} } \right)^{1/2} } \quad i=1,\ldots,m
\end{equation*}
where $\beta=1,2,4$ corresponds to the GOE, GUE, or GSE.  Then as $n \rightarrow \infty$,
\begin{equation*}
	\Prob[X_1 \leq \xi_1,\ldots, X_m \leq \xi_m] \longrightarrow \Phi_\Lambda(\xi_1,\ldots,\xi_m)
\end{equation*}
where $\Phi_\Lambda$ is the cdf for the $m$-dimensional normal distribution with covariance matrix $\Lambda_{i,j} = 1 - \frac{1}{\gamma} \max \{ \theta_k : i \leq k < j < m \}$ if $i < j$ and $\Lambda_{i,i} = 1$.  
\end{theorem}

\begin{remark}
The GUE ($\beta=2$) case in Theorems \ref{thm_bulk} and \ref{thm_edge} was shown by Gustavsson in \cite{gu}.
\end{remark}

\begin{remark} \label{rem_bulk_1d}
In the case $m=1$, Theorem \ref{thm_bulk} can be stated as follows.  
Set $t=t(k,n) = G^{-1}(k/n)$ where $k=k(n)$ is such that $k/n \rightarrow a \in (0,1)$ as $n \rightarrow \infty$.  If $x_k$ denotes eigenvalue number $k$ in the GOE, GUE, or GSE, it holds that, as $n \rightarrow \infty$,
\begin{equation*}
	\frac{x_k - t\sqrt{2n}}{\left( \frac{\log n}{2\beta(1-t^2)n} \right)^{1/2}} \longrightarrow N(0,1)
\end{equation*}
in distribution where $\beta=1,2,4$ corresponds to the GOE, GUE, or GSE.
\end{remark}

\begin{remark} \label{rem_edge_1d}
In the case $m=1$, Theorem \ref{thm_edge} can be stated as follows.  
Let $k$ be such that $k \rightarrow \infty$ but $\frac{k}{n} \rightarrow 0$ as $n \rightarrow \infty$ and let $x_{n-k}$ denote eigenvalue number $n-k$ in the GOE, GUE, or GSE.  Then it holds that, as $n \rightarrow \infty$,
\begin{equation*}
	\frac{ x_{n-k} - \sqrt{2n} \left( 1- \left( \frac{3 \pi k}{4 \sqrt{2} n} \right)^{2/3} \right) } { \left( \left( \frac{1}{12 \pi} \right)^{2/3} \frac{2 \log k}{\beta n^{1/3} k^{2/3}} \right)^{1/2} } \longrightarrow N(0,1)
\end{equation*}
in distribution where $\beta=1,2,4$ corresponds to the GOE, GUE, or GSE.
\end{remark}

\begin{remark}
One can omit the assumption that $k_i/n \rightarrow a_i$ in Theorem \ref{thm_bulk} and the conclusion still holds.  To see this, first consider the case $m=1$.  Let $x_k$ denote a sequence of eigenvalues from the bulk with $k=k(n)$ (where $k/n$ does not necessarily converge as $n \rightarrow \infty$).  Since $k/n < 1$, there exists a subsequence, say $k'=k(n_l)$, such that $k'/n_l \rightarrow a$ as $l \rightarrow \infty$ for some $a \in (0,1)$.  By Theorem \ref{thm_bulk}, the centered and scaled eigenvalues from the subsequence $x_{k'}$ converge to the standard normal distribution.  It follows that every subsequence has a further subsequence which converges in distribution to the standard normal.  Therefore, the entire sequence must converge in distribution to the standard normal.  

A similar argument allows one to omit the assumption that $k_i/n \rightarrow a_i$ in the case $m>1$.  
\end{remark}

\subsubsection{Results for Wigner Matrices}

In \cite{tao} and \cite{tao2}, Tao and Vu extend Gustavsson's results for the GUE to a more general class of Hermitian Wigner matrices.   

\begin{corollary}[Tao, Vu; Hermitian Wigner Matrices] \label{hermitian_wigner}
The conclusions of Theorems \ref{thm_bulk} and \ref{thm_edge} also hold with $\beta=2$ when $x_1 \leq x_2 \leq \ldots \leq x_n$ are the ordered eigenvalues of any other Wigner Hermitian matrix $M_n = \left(m_{ij}\right)_{1 \leq i,j \leq n}$ where the following moment conditions hold:
\begin{itemize}
	\item $\Re m_{ij}$ and $\Im m_{ij}$ have mean $0$ and variance $1/4$ for $1 \leq i < j \leq n$.
	\item $m_{ii}$ has mean $0$ and variance $1/2$ for $1 \leq i \leq n$.
	\item $\E((\Re m_{ij})^3) = \E((\Im m_{ij})^3) = 0$ for $1 \leq i < j \leq n$.
	\item $\E((\Re m_{ij})^4) = \E((\Im m_{ij})^4) = 3/16$ for $1 \leq i < j \leq n$.
\end{itemize}
\end{corollary}

In a similar fashion, we use Tao and Vu's Four Moment Theorem (see \cite{tao} and \cite{tao2}) to extend our results to a more general class of real symmetric Wigner matrices.

\begin{corollary}[Real Symmetric Wigner Matrices] \label{real_wigner_bulk}
The conclusions of Theorems \ref{thm_bulk} and \ref{thm_edge} also hold with $\beta=1$ when $x_1 \leq x_2 \leq \ldots \leq x_n$ are the ordered eigenvalues of any other real symmetric Wigner matrix $M_n = \left(m_{ij}\right)_{1 \leq i,j \leq n}$ where $m_{ij}$ has mean $0$ and variance $\frac{1+\delta_{ij}}{2}$ for $1 \leq i \leq j \leq n$ and $\E(m_{ij}^3) = 0$, $\E(m_{ij}^4) = 3/4$ for $1 \leq i < j \leq n$.  
\end{corollary}

The proof of Corollary \ref{real_wigner_bulk} is nearly identical to the proof of Corollary 19 in \cite{tao} and we omit the details here.  For the multidimensional cases, see Remark 20 in \cite{tao}. 

\begin{remark}
It is also possible to consider a quaternion self-dual Hermitian $n \times n$ matrix $M_n = \left(m_{jk}\right)_{1 \leq j,k \leq n}$ such that
\begin{align*}
	m_{jk} &= m_{jk}^{(0)} + m_{jk}^{(1)} e_1 + m_{jk}^{(2)} e_2 + m_{jk}^{(3)} e_3, \qquad 1 \leq j < k \leq n, \\
	m_{jj} &= m_{jj}^{(0)}, \qquad 1 \leq j \leq n
\end{align*}
where $\{ m_{jk}^{(i)} : 1 \leq j < k \leq n, i=0,1,2,3 \}$ are i.i.d. real random variables and $\{m_{jj}^{(0)}: 1 \leq j \leq n \}$ are i.i.d. real random variables.  Such an ensemble of matrices would generalize the GSE, but historically have not been studied.  
\end{remark}

\begin{remark}
In order for eigenvalues $x_k$ and $x_m$ in the bulk to be independent in the limit, it must be the case that $\left| k - m \right| \sim n$.
\end{remark}

\section{Limiting Distribution of a Single Eigenvalue in the GOE and GSE}

In this section, we will prove Theorems \ref{thm_bulk} and \ref{thm_edge} for the GOE and GSE in the case $m=1$ (see Remarks \ref{rem_bulk_1d} and \ref{rem_edge_1d}).  Although we prove the general case for any $m \geq 1$ in Section \ref{multi_dim_proof}, we have found it instructive to start with the one-dimensional case.

\subsection{A Central Limit Theorem}

In the proof of the GUE case of Theorems \ref{thm_bulk} and \ref{thm_edge}, Gustavsson relies on the fact that the GUE defines a determinantal random point process.  Gustavsson utilizes a theorem due to Costin, Lebowitz, and Soshnikov (\cite{cl}, \cite{pe}, and \cite{so2}).  Let $\#_{\GUE_n}(I)$ denote the number of eigenvalues (from an $n \times n$ matrix drawn from the GUE) in the subset $I \subset \R$.  

\begin{theorem}[Costin-Lebowitz, Soshnikov] \label{costin_lebowitz} 
If $\var(\#_{\GUE_n}(I_n)) \rightarrow \infty$ as $n \rightarrow \infty$, then
\begin{equation*}
	\frac{ \#_{\GUE_n}(I_n) - \E[\#_{\GUE_n}(I_n)] } {\sqrt{\var(\#_{\GUE_n}(I_n))}} \longrightarrow N(0,1)
\end{equation*}
in distribution as $n \rightarrow \infty$.  
\end{theorem}

\begin{remark}
We stated the theorem here in terms of the GUE, but the result is actually more general and holds for any sequence of determinantal random point fields.  We state the more general version of this result and give a proof in Appendix \ref{section:clt} (see Theorem \ref{general_costin_lebowitz}). 
\end{remark}

We begin by proving a version of Theorem \ref{costin_lebowitz} for the GOE.  To do this, we utilize the fact that Gustavsson already proved the GUE case of Theorems \ref{thm_bulk} and \ref{thm_edge} in \cite{gu} and a result by Forrester and Rains in \cite{fo} that relates the eigenvalues of the different ensembles.  
\begin{theorem}[Forrester-Rains] \label{forrester_rains}
The following relations hold between matrix ensembles:
\begin{align*}
	\GUE_n &= \mathrm{even}(\GOE_n \cup \GOE_{n+1}) \\
	\GSE_n &= \mathrm{even}(\GOE_{2n+1}) \cdot \frac{1}{\sqrt{2}}
\end{align*}
\end{theorem}

\begin{remark}
The result by Forrester and Rains in \cite{fo} is actually much more general.  Here we only consider two specific cases. 
\end{remark}

\begin{remark}
The multiplication by $\frac{1}{\sqrt{2}}$ denotes scaling the $(2n+1) \times (2n+1)$ GOE matrix by a factor of $\frac{1}{\sqrt{2}}$.  
\end{remark}

\begin{remark}
The first statement can be interpreted in the following way.  Take two independent matrices from the GOE: one of size $n \times n$ and one of size $(n+1) \times (n+1)$.  Superimpose the eigenvalues on the real line to form a random point process with $2n+1$ particles.  Then the new random point process formed by taking the $n$ even particles has the same distribution as the eigenvalues of an $n \times n$ matrix from the GUE.  
\end{remark}

\begin{remark}
The first relation was originally conjectured in 1962 by Dyson for the circular unitary ensemble and the circular orthogonal ensemble (see \cite{dy}).  It was proven the same year by Gunson in \cite{gun}.  
\end{remark}

We will also need the following result.  
\begin{lemma} \label{lemma_tight}
	Let $\{X_n\}$ and $\{Y_n\}$ be sequences of random variables where $X_n$ and $Y_n$ are i.i.d for each $n \in \N$.  If $X_n + Y_n\longrightarrow \mathrm{N}(0,2)$ in distribution, then $X_n \longrightarrow \mathrm{N}(0,1)$ in distribution. 
\end{lemma}
\begin{proof}
We wish to show that $\{X_n\}$ is tight and that every subsequence $\{X_{n_k}\}$ has a further subsequence $\{X_{n_{k_l}}\}$ such that $X_{n_{k_l}} \longrightarrow \mathrm{N}(0,1)$ in distribution as $l \rightarrow \infty$.  We proceed as follows:
\begin{itemize}
	\item We will show that $\{X_n\}$ and $\{Y_n\}$ are tight.  Notice that since $X_n$ and $Y_n$ are i.i.d for each $n \in \N$, it is enough to just show $\{X_n\}$ is tight.  
	\item Assuming $\{X_n\}$ is tight, we can conclude that every subsequence $\{X_{n_k}\}$ has a further subsequence $\{X_{n_{k_l}}\}$ that converges in distribution.  Since $X_n$ and $Y_n$ are i.i.d, we have that
	\begin{equation*}
		\E\left[ e^{it (X_{n_{k_l}} + Y_{n_{k_l}})} \right] = \left( \E \left[ e^{itX_{n_{k_l}}} \right] \right)^2 \longrightarrow e^{-t^2} \text{ as } l \rightarrow \infty,
	\end{equation*}
by assumption.  Thus, we can conclude that every subsequence $\{X_{n_k}\}$ has a further subsequence $\{X_{n_{k_l}}\}$ that converges in distribution to $\mathrm{N}(0,1)$.  
	\item This would complete the proof, for if $X_n \not \rightarrow \mathrm{N}(0,1)$ in distribution, then there exists $\epsilon > 0$, $t \in \R$, and a subsequence $\{X_{n_k}\}$ such that
	\begin{equation*}
		\left| \E\left[ e^{it X_{n_k}} \right] - e^{-\frac{t^2}{2}} \right| > \epsilon.
	\end{equation*}
But this is a contradiction since there is a further subsequence $\{X_{n_{k_l}}\}$ that converges in distribution to $\mathrm{N}(0,1)$. 
\end{itemize}

All that remains is to show that $\{X_n\}$ is tight.  Let $\epsilon > 0$.  By taking both $M>0$ and $n>N$ large, 

\begin{align*}
	\epsilon > \Prob\left(X_n+Y_n > M\right) \geq \Prob\left(X_n>\frac{M}{2}, Y_n > \frac{M}{2}\right) = \left[ \Prob\left(X_n > \frac{M}{2} \right) \right]^2.
\end{align*}

Similarly, 
\begin{align*}
	\epsilon > \Prob\left(X_n+Y_n < -M\right) \geq \Prob\left(-X_n>\frac{M}{2}, -Y_n > \frac{M}{2}\right) = \left[ \Prob\left(-X_n > \frac{M}{2} \right) \right]^2.
\end{align*}

Thus, 
\begin{equation*}
	\Prob\left(\left|X_n\right| > \frac{M}{2}\right) \leq 2\sqrt{\epsilon} \text{ for all } n>N
\end{equation*}
and the result follows.
\end{proof}

\begin{lemma} \label{GOE_clt}
If $\var(\#_{\GUE_n}(I_n)) \rightarrow \infty$ as $n \rightarrow \infty$, then
\begin{equation*}
	\frac{ \#_{\GOE_n}(I_n) - \E[\#_{\GOE_n}(I_n)] } {\sqrt{2\var(\#_{\GUE_n}(I_n))}} \longrightarrow N(0,1)
\end{equation*}
in distribution as $n \rightarrow \infty$.  
\end{lemma}

\begin{proof}
By Theorem \ref{forrester_rains}, we have that
\begin{equation*}
	\#_{\GUE_n}(I_n) = \frac{1}{2} \left[ \#_{\GOE_n}(I_n) + \#_{\GOE_{n+1}}(I_n) + \xi_n(I_n) \right]
\end{equation*}
where $\xi_n(I_n)$ takes values in $\{-1,0,1\}$.  Thus by Cauchy's interlacing theorem (see Lemma \ref{cauchy_interlace} in Appendix \ref{section:interlace}), we can write,
\begin{equation} \label{clt_start}
	\#_{\GUE_n}(I_n) = \frac{1}{2} \left[ \#_{\GOE_n}(I_n) + \#_{\GOE_{n}'}(I_n) + \xi'_n(I_n) \right]
\end{equation}
where we obtain $\GOE_{n}'$ from $\GOE_{n+1}$ by considering the principle submatrix of $\GOE_{n+1}$ and $\xi'_n(I_n)$ takes values in $\{-2,-1,0,1,2\}$.  Note that $ \#_{\GOE_n}(I_n)$ and $\#_{\GOE_{n}'}(I_n)$ are independent because $\GOE_{n+1}$ and $\GOE_n$ denote independent matrices from the GOE.  By taking expectation on both sides of \eqref{clt_start} we obtain
\begin{equation} \label{GOE_expectation_long}
	\E[\#_{\GUE_n}(I_n)] = \frac{1}{2} \left[ \E[\#_{\GOE_n}(I_n)] + \E[\#_{\GOE_{n}'}(I_n)] + \E[\xi'_n(I_n)] \right].
\end{equation}
Finally we subtract the expectation and divide by the standard deviation on both sides of \eqref{clt_start} to obtain
\begin{align*}
	\sqrt{2} \frac { \#_{\GUE_n}(I_n) - \E[\#_{\GUE_n}(I_n)] } { \sqrt{\var(\#_{\GUE_n}(I_n))} }   = &  \frac{ \#_{\GOE_n}(I_n) - \E[\#_{\GOE_n}(I_n)] } { \sqrt{2\var(\#_{\GUE_n}(I_n))} }  \\
	& + \frac{ \#_{\GOE_{n}'}(I_n) - \E[\#_{\GOE_{n}'}(I_n)] } { \sqrt{2\var(\#_{\GUE_n}(I_n))} }  \\
	& + \frac{ \xi'_n(I_n) -  \E[\xi'_n(I_n)] } { \sqrt{2\var(\#_{\GUE_n}(I_n))} }  \\
	= & X_n + Y_n + \epsilon_n.
\end{align*}
The left hand side converges to $N(0,2)$ by Theorem \ref{costin_lebowitz} and 
\begin{equation*}
	\left| \epsilon_n \right| \leq \frac{4}{\sqrt{2\var(\#_{\GUE_n}(I_n))}} \longrightarrow 0 \text{ almost surely as } n \rightarrow \infty.
\end{equation*}
Therefore by Lemma \ref{lemma_tight}, $X_n \longrightarrow N(0,1)$ in distribution as $n \rightarrow \infty$.
\end{proof}

\begin{remark}
As a consequence of equation \eqref{GOE_expectation_long}, we have that for any subset $I \subset \R$,
\begin{equation} \label{GOE_expectation}
	\E[\#_{\GUE_n}(I)] = \E[\#_{\GOE_n}(I)] + O(1).
\end{equation}
\end{remark}

\subsection{Gustavsson's Calculations for the GUE}

We will also need some calculations provided by Gustavsson in the following lemmas.
\begin{lemma}[Gustavsson] \label{GUE_bulk_expectation}
Let $t=t(k,n)$ be the solution to the equation
\begin{equation*}
	n\frac{2}{\pi} \int_{-1}^t \sqrt{1-x^2} \udx = k
\end{equation*}
where $k=k(n)$ is such that $k/n \rightarrow a \in (0,1)$ as $n \rightarrow \infty$.  The expected number of eigenvalues from the GUE in the interval
\begin{equation*}
	I_n = \left[ \sqrt{2n}t+x\sqrt{\frac{\log n}{2n}}, \infty \right)
\end{equation*}
is given by
\begin{equation*}
	\E[\#_{\GUE_n}(I_n)] = n-k-\frac{x}{\pi}\sqrt{(1-t^2)\log n} + O\left( \frac{\log n}{n} \right).
\end{equation*}
\end{lemma}

\begin{lemma}[Gustavsson] \label{GUE_edge_expectation}
The expected number of eigenvalues in the interval $I_n = [ \sqrt{2n}t, \infty)$, where $t \rightarrow 1^{-}$ as $n \rightarrow \infty$, is given by 
\begin{equation*}
	\E[\#_{\GUE_n}(I_n)] = \frac{4 \sqrt{2}}{3 \pi} n(1-t)^{3/2} + O(1).
\end{equation*}
\end{lemma}

\begin{lemma}[Gustavsson] \label{GUE_variance}
Let $\delta>0$ and suppose that $t$, which may depend on $n$, is such that $-1 + \delta \leq t < 1$ and $n(1-t)^{3/2} \rightarrow \infty$ as $n \rightarrow \infty$.  Then the variance of the number of eigenvalues from the GUE in the interval $I_n = [t\sqrt{2n},\infty)$ is given by 
\begin{equation*}
	\var(\#_{\GUE_n}(I_n)) = \frac{1}{2\pi^2} \log [n(1-t)^{3/2}](1+\eta(n))
\end{equation*}
where $\eta(n) \rightarrow 0$ as $n \rightarrow \infty$.  
\end{lemma}

\subsection{Proof of Main Results}

We now prove the main results.   
\begin{proof}[Proof of Theorem \ref{thm_bulk} for the GOE]
Set 
\begin{equation*}
	I_n = \left[ t\sqrt{2n} + \xi \left(\frac{\log n}{2(1-t^2)n} \right)^{1/2}, \infty \right). 
\end{equation*}
By Lemma \ref{GUE_bulk_expectation} and equation \eqref{GOE_expectation} we can take $x = \frac{\xi}{\sqrt{1-t^2}}$ and obtain
\begin{align*}
	\E[\#_{\GOE_n}(I_n)] &= n-k-\frac{x}{\pi}\sqrt{(1-t^2)\log n} + O\left( \frac{\log n}{n} \right) + O(1) \\
					&= n-k-\frac{\xi}{\pi}\sqrt{\log n} + O(1).
\end{align*}
Combining this with Lemma \ref{GUE_variance} we get
\begin{align*}
	\Prob \left[ \frac{x_k - t\sqrt{2n}}{\left( \frac{\log n}{2(1-t^2)n} \right)^{1/2}} \leq \xi \right] &= \Prob \left[ x_k \leq t\sqrt{2n} + \xi \left( \frac{\log n}{2(1-t^2)n} \right)^{1/2} \right] \\
	&= \Prob[ \#_{GOE_n}(I_n) \leq n-k] \\
	&= \Prob \left[ \frac{ \#_{\GOE_n}(I_n) - \E[\#_{\GOE_n}(I_n)] } { \sqrt{2\var(\#_{\GUE_n}(I_n))} } \leq \frac{ n-k- \E[\#_{\GOE_n}(I_n)] } {  \sqrt{2\var(\#_{\GUE_n}(I_n))} } \right] \\
	&= \Prob \left[ \frac{ \#_{\GOE_n}(I_n) - \E[\#_{\GOE_n}(I_n)] } { \sqrt{2\var(\#_{\GUE_n}(I_n))} } \leq \xi + \epsilon(n) \right]
\end{align*}
where $\epsilon(n) \rightarrow 0$ as $n \rightarrow \infty$.  By Lemma \ref{GOE_clt} the conclusion follows.
\end{proof}

\begin{proof}[Proof of Theorem \ref{thm_edge} for the GOE]
Set
\begin{equation*}
	I_n = \left[ \sqrt{2n}\left( 1 - \left( \frac{3 \pi k}{4 \sqrt{2} n} \right)^{2/3} \right) + \left( \left( \frac{1}{12 \pi} \right)^{2/3} \frac{2 \log k}{n^{1/3} k^{2/3}} \right)^{1/2} \xi, \infty \right).
\end{equation*}
By Lemma \ref{GUE_edge_expectation} and equation \eqref{GOE_expectation} we have that
\begin{equation*}
	\E[\#_{\GOE_n}(I_n)] = \frac{4 \sqrt{2} }{3 \pi} n(1-t)^{3/2} + O(1)
\end{equation*}
where 
\begin{equation*}
	t = 1 - \left( \frac{3 \pi k}{4 \sqrt{2} n} \right)^{2/3}  + \frac{1}{\sqrt{n}} \left( \left( \frac{1}{12 \pi} \right)^{2/3} \frac{ \log k}{n^{1/3} k^{2/3}} \right)^{1/2} \xi.
\end{equation*}
Combining this with Lemma \ref{GUE_variance} we get
\begin{align*}
	\Prob &\left[ \frac{ x_{n-k} - \sqrt{2n} \left( 1- \left( \frac{3 \pi k}{4 \sqrt{2} n} \right)^{2/3} \right) } { \left( \left( \frac{1}{12 \pi} \right)^{2/3} \frac{2 \log k}{n^{1/3} k^{2/3}} \right)^{1/2} } \leq \xi \right] \\
	 &\qquad = \Prob \Bigg[ x_{n-k} \leq \sqrt{2n}\left( 1 - \left( \frac{3 \pi k}{4 \sqrt{2} n} \right)^{2/3} \right) +  \left( \left( \frac{1}{12 \pi} \right)^{2/3} \frac{2 \log k}{n^{1/3} k^{2/3}} \right)^{1/2} \xi \Bigg] \\
	&\qquad = \Prob \left[ \#_{GOE_n}(I_n) \leq k \right] \\
	&\qquad = \Prob \left[ \frac{ \#_{\GOE_n}(I_n) - \E[\#_{\GOE_n}(I_n)] } { \sqrt{2\var(\#_{\GUE_n}(I_n))} } \leq \frac{ k- \E[\#_{\GOE_n}(I_n)] } {  \sqrt{2\var(\#_{\GUE_n}(I_n))} } \right] \\
	 &\qquad = \Prob \left[ \frac{ \#_{\GOE_n}(I_n) - \E[\#_{\GOE_n}(I_n)] } { \sqrt{2\var(\#_{\GUE_n}(I_n))} } \leq \xi + \epsilon(n) \right]
\end{align*}
where $\epsilon(n) \rightarrow 0$ as $n \rightarrow \infty$.  By Lemma \ref{GOE_clt} the conclusion follows.
\end{proof}

\begin{proof}[Proof of Theorems \ref{thm_bulk} and \ref{thm_edge} for the GSE]
Let $x_1 < x_2 < \ldots < x_n$ denote the ordered eigenvalues of an $n \times n$ matrix from the GSE and let $y_1 < y_2 < \ldots y_{2n+1}$ denote the ordered eigenvalues of an $(2n+1) \times (2n+1)$ matrix from the GOE.  By Theorem \ref{forrester_rains} it follows that $x_k = \frac{y_{2k}}{\sqrt{2}}$ in distribution and hence the result follows by the GOE case of Theorems \ref{thm_bulk} and \ref{thm_edge}.  
\end{proof}

\section{Joint Limiting Distribution of Several Eigenvalues in the GOE and GSE} \label{multi_dim_proof}

\subsection{A Multidimensional Central Limit Theorem}
For the multidimensional case, we will need the following theorem, \cite{so3}:

\begin{theorem}[Soshnikov] \label{costin_lebowitz_multi}
Let $\{I_n^{(1)},\ldots, I_n^{(k)}\}_{n=1}^\infty$ be a family of Borel subsets of $\R$, disjoint for any fixed $n$, with compact closure.  Suppose
\begin{equation*}
	\var \left( \sum_{j=1}^k \alpha_j \#_{\GUE_n}\left(I_n^{(j)}\right) \right) \quad \alpha_1,\ldots,\alpha_k \in \R
\end{equation*}
grows to infinity with $n$ in such a way that
\begin{equation} \label{clt_var_eq}
	\var \left( \#_{\GUE_n}\left(I_n^{(i)}\right) \right) = O \left( \var \left( \sum_{j=1}^k \alpha_j \#_{\GUE_n}\left(I_n^{(j)}\right) \right) \right)
\end{equation}
for any $1 \leq i \leq k$.  Then the central limit theorem holds:
\begin{equation*} 
	\frac{ \sum_{j=1}^k \alpha_j \#_{\GUE_n}\left(I_n^{(j)}\right) - \E \left[ \sum_{j=1}^k \alpha_j \#_{\GUE_n}\left(I_n^{(j)}\right) \right] }  { \sqrt{ \var \left( \sum_{j=1}^k \alpha_j \#_{\GUE_n}\left(I_n^{(j)}\right) \right) } } \longrightarrow N(0,1)
\end{equation*}
in distribution.
\end{theorem}

\begin{remark}
The theorem in \cite{so3} is more general than the theorem stated here.  We state a more general version of this result and give a proof in Appendix \ref{section:clt} (see Theorem \ref{general_costin_lebowitz_multi}).  
\end{remark}

\begin{remark}
In general, if $\{X_n^{(1)},\ldots,X_n^{(k)}\}_{n=1}^\infty$ is a family of random variables and 
\begin{equation*}
	\frac{ \sum_{j=1}^k \alpha_j X_n^{(j)} - \E\left[ \sum_{j=1}^k \alpha_j X_n^{(j)} \right] } { \left( \var \left( \sum_{j=1}^k \alpha_j X_n^{(j)} \right) \right)^{1/2} }
\end{equation*}
converges to a normal distribution as $n \rightarrow \infty$ for all $\alpha_1,\ldots,\alpha_k \in \R$, then $X_n^{(1)},\ldots,X_n^{(k)}$ are jointly normally distributed in the limit, \cite{gut}.
\end{remark}

\begin{remark}
If \eqref{clt_var_eq} holds for every $\alpha_1,\ldots,\alpha_k \in \R$, then the random variables
\begin{equation*}
	\#_{\GUE_n}\left(I_n^{(1)}\right),\ldots, \#_{\GUE_n}\left(I_n^{(k)}\right)
\end{equation*}
are jointly normally distributed in the limit.  
\end{remark}

For the GOE, we will prove the following lemma.

\begin{lemma} \label{GOE_clt_multi}
Let $\{I_n^{(1)},\ldots, I_n^{(k)}\}_{n=1}^\infty$ be a family of Borel subsets of $\R$, disjoint for any fixed $n$, with compact closure.  Suppose
\begin{equation*}
	\var \left( \sum_{j=1}^k \alpha_j \#_{\GUE_n}\left(I_n^{(j)}\right) \right) \quad \alpha_1,\ldots,\alpha_k \in \R
\end{equation*}
grows to infinity with $n$ in such a way that
\begin{equation} \label{clt_var_GOE_eq}
	\var \left( \#_{\GUE_n}\left(I_n^{(i)}\right) \right) = O \left( \var \left( \sum_{j=1}^k \alpha_j \#_{\GUE_n}\left(I_n^{(j)}\right) \right) \right)
\end{equation}
for any $1 \leq i \leq k$.  Then for the GOE:
\begin{equation*} 
	\frac{ \sum_{j=1}^k \alpha_j \#_{\GOE_n}\left(I_n^{(j)}\right) - \E \left[ \sum_{j=1}^k \alpha_j \#_{\GOE_n}\left(I_n^{(j)}\right) \right] }  { \sqrt{ 2 \var \left( \sum_{j=1}^k \alpha_j \#_{\GUE_n}\left(I_n^{(j)}\right) \right) } } \longrightarrow N(0,1)
\end{equation*}
in distribution.
\end{lemma}

\begin{proof}
By following the proof of Lemma \ref{GOE_clt}, we can write
\begin{align} \label{clt_start_multi}
	\begin{split}
	\sum_{j=1}^k &\alpha_j \#_{\GUE_n}\left(I_n^{(j)}\right) = 
	 \frac{1}{2} \sum_{j=1}^k \alpha_j \left( \#_{\GOE_n}\left(I_n^{(j)}\right) + \#_{\GOE_{n}'}\left(I_n^{(j)}\right) + \xi'_n\left(I_n^{(j)}\right) \right)
	\end{split}
\end{align}
where $\xi'_n\left(I_n^{(j)}\right)$ takes values in $\{-2,-1,0,1,2\}$.  Define
\begin{align*}
	X_n = & \frac{ \sum_{j=1}^k \alpha_j \left( \#_{\GOE_n}\left(I_n^{(j)}\right) - \E\left[ \#_{\GOE_n}\left(I_n^{(j)}\right) \right] \right) } { \sqrt{ 2 \var \left( \sum_{j=1}^k \alpha_j \#_{\GUE_n}\left(I_n^{(j)}\right) \right) } }, \\
	Y_n = & \frac{ \sum_{j=1}^k \alpha_j \left( \#_{\GOE_{n}'}\left(I_n^{(j)}\right) - \E\left[ \#_{\GOE_{n}'}\left(I_n^{(j)}\right) \right] \right) } { \sqrt{ 2 \var \left( \sum_{j=1}^k \alpha_j \#_{\GUE_n}\left(I_n^{(j)}\right) \right) } }, \\
	\epsilon_n = & \frac{ \sum_{j=1}^k \alpha_j \xi'_n\left(I_n^{(j)}\right) - \E \left[ \sum_{j=1}^k \alpha_j \xi'_n\left(I_n^{(j)}\right) \right] } { \sqrt{ 2 \var \left( \sum_{j=1}^k \alpha_j \#_{\GUE_n}\left(I_n^{(j)}\right) \right) } }.
\end{align*}
Notice that for each $n \in \N$, $X_n$ and $Y_n$ are i.i.d.  By equation \eqref{clt_start_multi} and Theorem \ref{costin_lebowitz_multi}, we have that $X_n + Y_n + \epsilon_n$ is equal to
\begin{align*}
	 \sqrt{2} \frac{ \sum_{j=1}^k \alpha_j \#_{\GUE_n}\left(I_n^{(j)}\right) - \E \left[ \sum_{j=1}^k \alpha_j \#_{\GUE_n}\left(I_n^{(j)}\right) \right] } { \sqrt{ \var \left( \sum_{j=1}^k \alpha_j \#_{\GUE_n}\left(I_n^{(j)}\right) \right) } } \longrightarrow N(0,2)
\end{align*}
where the equality is everywhere and the convergence is in distribution.  Since
\begin{equation*}
	\left| \epsilon_n \right|  \leq \frac{4 \sum_{j=1}^k \alpha_j}{ \sqrt{ 2 \var \left( \sum_{j=1}^k \alpha_j \#_{\GUE_n}\left(I_n^{(j)}\right) \right) } } \longrightarrow 0 \text{ as } n \rightarrow \infty
\end{equation*}
almost surely, Lemma \ref{lemma_tight} implies that $X_n \longrightarrow N(0,1)$ in distribution as $n \rightarrow \infty$.  
\end{proof}

\begin{remark}
If \eqref{clt_var_GOE_eq} holds for every $\alpha_1,\ldots,\alpha_k \in \R$, then the random variables
\begin{equation*}
\#_{\GOE_n}\left(I_n^{(1)}\right),\ldots, \#_{\GOE_n}\left(I_n^{(k)}\right)
\end{equation*}
are jointly normally distributed in the limit.  
\end{remark}

\subsection{Proof of Main Results}

\begin{proof}[Proof of Theorem \ref{thm_bulk} for the GOE]
Let $k_i$, $s_i$, $\theta_i$, and $X_i$ as in the formulation of Theorem \ref{thm_bulk}.  Let $\xi_1,\ldots,\xi_m \in \R$ and define
\begin{align*}
	I_n^{(1)} &= \left( s_1 \sqrt{2n}+\xi_1 \left( \frac{ \log n } { 2 (1-s_1^2)n} \right)^{1/2},\infty \right), \\
	I_n^{(i)} &= \Bigg( s_i \sqrt{2n}+\xi_i \left( \frac{ \log n } { 2 (1-s_i^2)n} \right)^{1/2}, s_{i-1} \sqrt{2n}+\xi_{i-1} \left( \frac{ \log n } { 2 (1-s_{i-1}^2)n} \right)^{1/2} \Bigg]
\end{align*}
for $2 \leq i \leq m$.  For convenience, let
\begin{align*}
	S_{n,k} &= \sum_{j=1}^k \#_{\GOE_n}\left(I_n^{(j)}\right), \\
	\sigma_{n,k}^2 &= 2 \var \left( \sum_{j=1}^k \#_{\GUE_n}\left(I_n^{(j)}\right) \right)
\end{align*}
for $1 \leq k \leq m$.  Then we have that (for $n$ large enough)
\begin{align*}
	\Prob &[X_1 \leq \xi_1,\ldots,X_m \leq \xi_m] =
	  \Prob \left[ \frac{S_{n,l} - \E[S_{n,l}] }{ \sigma_{n,l}} \leq \frac{n-k_l - \E[S_{n,l}] } { \sigma_{n,l} }, 1 \leq l \leq m \right]
\end{align*}
We now need to show that the random variables
\begin{equation*}
	\#_{\GOE_n}\left(I_n^{(1)}\right), \#_{\GOE_n}\left(I_n^{(1)}\right)+\#_{\GOE_n}\left(I_n^{(2)}\right),\ldots,\sum_{j=1}^m \#_{\GOE_n}\left(I_n^{(j)}\right)
\end{equation*}
are jointly normal in the limit.  To do so, we will use Lemma \ref{GOE_clt_multi} and show that all linear combinations of the variables are normally distributed in the limit.  This is equivalent to showing that the random variables
\begin{equation*}
	\#_{\GOE_n}\left(I_n^{(1)}\right), \#_{\GOE_n}\left(I_n^{(2)}\right),\ldots,\#_{\GOE_n}\left(I_n^{(m)}\right)
\end{equation*}
are jointly normal in the limit.  Let $\alpha_1,\ldots,\alpha_m \in \R$ with $\alpha_1^2 + \cdots + \alpha_m^2 \neq 0$.  In \cite{gu}, Gustavsson showed that \eqref{clt_var_GOE_eq} holds for our choice of intervals $I_n^{(1)},\ldots,I_n^{(m)}$.  In fact, Gustavsson showed that the variance is of magnitude $\log n$.  Therefore the result follows by Lemma \ref{GOE_clt_multi}.  

To complete the proof, we will calculate the correlations between the random variables
\begin{equation*}
	\#_{\GOE_n}\left(I_n^{(1)}\right), \#_{\GOE_n}\left(I_n^{(1)}\right)+\#_{\GOE_n}\left(I_n^{(2)}\right),\ldots,\sum_{j=1}^m \#_{\GOE_n}\left(I_n^{(j)}\right).
\end{equation*}
If $j<i$, we have that $s_j - s_i \sim n^{-\gamma}$ where $\gamma = 1 - \max_{j \leq k < i} \theta_k$.  Then Gustavsson showed that for the GUE,
\begin{align*}
	\var & \left( \sum_{k=1}^i \#_{\GUE_n}\left(I_n^{(k)}\right) - \sum_{k=1}^j \#_{\GUE_n}\left(I_n^{(k)}\right) \right)  \\
	&= \var \left( \#_{\GUE_n}\left( \bigcup_{k=j+1}^i I_n^{(k)}\right) \right) = \frac{1-\gamma}{\pi^2} \log n + O(\log \log n) \text{ and } \\
	\var & \left( \sum_{k=1}^l \#_{\GUE_n}\left(I_n^{(k)}\right) \right) = \frac{1}{2 \pi^2} \log n + O(\log \log n)
\end{align*}
for any $1 \leq l \leq m$.  Also, by Theorem \ref{forrester_rains}, we have the following relation between the GOE and GUE
\begin{equation} \label{GUE_GOE_var_eq}
	\var \left( \#_{\GOE_n}\left(I_n^{(k)}\right) \right) = 2 \var \left( \#_{\GUE_n}\left(I_n^{(k)}\right) \right) + o(\log n)
\end{equation}
for any $1 \leq k \leq m$.  Thus we have that the correlation $\rho$ is given by
\begin{equation*}
	\rho(S_{n,i},S_{n,j}) = \frac{ \frac{1}{2} \left( \var (S_{n,i}) + \var (S_{n,j}) - \var(S_{n,i} - S_{n,j}) \right) } { \sqrt{ \var (S_{n,i}) \var (S_{n,j}) } } = \gamma + o(1).
\end{equation*}
\end{proof}

\begin{proof}[Proof of Theorem \ref{thm_edge} for the GOE]
This proof is very similar to the proof of the GOE case of Theorem \ref{thm_bulk}.  In this case, the intervals are given by
\begin{align*}
	I_n^{(1)} &= \left( \sqrt{2n} \left( 1 - C_1 \left(\frac{k_1}{n} \right)^{2/3} \right) + \xi_1 C_2 \left( \frac{ 2 \log k_1 }{ n^{1/3} k_1^{2/3}} \right)^{1/2}, \infty \right), \\
	I_n^{(i)} &= \Bigg( \sqrt{2n} \left( 1 - C_1 \left(\frac{k_i}{n} \right)^{2/3} \right) + \xi_i C_2 \left( \frac{ 2 \log k_i }{ n^{1/3} k_i^{2/3}} \right)^{1/2}, \\
	& \qquad \qquad \sqrt{2n} \left( 1 - C_1 \left(\frac{k_{i-1}}{n} \right)^{2/3} \right) + \xi_{i-1} C_2 \left( \frac{ 2 \log k_{i-1} }{ n^{1/3} k_{i-1}^{2/3}} \right)^{1/2} \Bigg]
\end{align*}
where $C_1$, $C_2$ are known constants and $2 \leq i \leq m$.  For sufficiently large $n$, the sets $I_n^{(1)},\ldots,I_n^{(m)}$ are intervals.  We will now prove that 
\begin{equation*}
	\#_{\GOE_n}\left(I_n^{(1)}\right), \#_{\GOE_n}\left(I_n^{(2)}\right),\ldots,\#_{\GOE_n}\left(I_n^{(m)}\right)
\end{equation*}
are jointly normally distributed in the limit as $n \rightarrow \infty$.  

In \cite{gu}, Gustavsson showed that \eqref{clt_var_GOE_eq} holds for our choice of intervals $I_n^{(1)},\ldots,I_n^{(m)}$.  In fact, Gustavsson showed that the variance is again of magnitude $\log n$ for any $\alpha_1^2 + \cdots + \alpha_m^2 \neq 0$.  Therefore the limiting distribution is normal by Lemma \ref{GOE_clt_multi}.

The calculations of the correlations is similar to the calculations in the GOE case of Theorem \ref{thm_bulk} and follow from Gustavsson's calculations for the GUE and equation \eqref{GUE_GOE_var_eq}.  
\end{proof}

\begin{proof}[Proof of Theorems \ref{thm_bulk} and \ref{thm_edge} for the GSE]
Let $x_1 < x_2 < \ldots < x_n$ denote the ordered eigenvalues of an $n \times n$ matrix from the GSE and let $y_1 < y_2 < \ldots y_{2n+1}$ denote the ordered eigenvalues of an $(2n+1) \times (2n+1)$ matrix from the GOE.  By Theorem \ref{forrester_rains} it follows that the joint distribution of $x_{k_1},\ldots,x_{k_m}$ is equal to the joint distribution of $\frac{y_{2k_1}}{\sqrt{2}},\ldots,\frac{y_{2k_m}}{\sqrt{2}}$.  Therefore the result follows by the GOE case of Theorems \ref{thm_bulk} and \ref{thm_edge}.  
\end{proof}

\appendix

\section{Interlacing Theorem} \label{section:interlace}
  The interlacing theorem we require is known as Cauchy's interlacing theorem for eigenvalues of Hermitian matrices (see \cite{fi}).  Recall that if two polynomials $f(x)$ and $g(x)$ have real roots $r_1 \leq r_2 \leq \ldots \leq r_n$ and $s_1 \leq s_2 \leq \ldots \leq s_{n-1}$, then we say that $f$ and $g$ interlace if 
\begin{equation*}
	r_1 \leq s_1 \leq r_2 \leq s_2 \leq \ldots \leq s_{n-1} \leq r_n
\end{equation*}

\begin{lemma}[Cauchy's Interlacing Theorem] \label{cauchy_interlace}
If $A$ is a Hermitian matrix and $B$ is a principle submatrix of $A$, then the eigenvalues of $B$ interlace with the eigenvalues of $A$.  
\end{lemma}

\section{Central Limit Theorems} \label{section:clt}
In this section, we will state and prove two central limit theorems for determinantal random point fields.  Let $\{ \mathcal{P}_t \}_{t \geq 0}$ be a family of random point fields on $\R^d$ such that their correlation functions $\rho_{t,k}$ have the determinantal form
\begin{equation*}
	\rho_{t,k}(x_1,\ldots,x_k) = \det \left( K_t(x_i, i_j) \right)_{1 \leq i,j \leq k}
\end{equation*}
where $K_t(x,y)$ is a Hermitian kernel.  Let $\{ I_t \}_{t \geq 0}$ be a collection of Borel subsets in $\R^d$ and let $A_t: L^2(I_t) \rightarrow L^2(I_t)$ denote an integral operator on $I_t$ with kernel $K_t$.  Define $\nu_t$ to be the number of particles in $I_t$, i.e. $\nu_t = \#(I_t)$.  Let $\E_t$ and $\var_t$ be the expectation and varaince with respect to the probability distribution of the random point field $\mathcal{P}_t$.  

\begin{theorem}[Costin-Lebowitz, Soshnikov] \label{general_costin_lebowitz} 
Let $A_t = K_t \cdot \chi_{I_t}$ be a family of trace class Hermitian operators associated with determinantal random point fields $\{ \mathcal{P}_t \}$ such that $\var_t (\nu_t) = \tr{\left(A_t - A_t^2\right)}$ goes to inifinity as $t \rightarrow \infty$.  Then
\begin{equation} \label{general_clt_eq}
	\frac{ \nu_t - \E_t[\nu_t] } { \sqrt{ \var_t (\nu_t) } } \longrightarrow N(0,1)
\end{equation}
in distribution as $t \rightarrow \infty$.  
\end{theorem}

\begin{remark}
The result was proven by Costin and Lebowtiz in \cite{cl} for the case when $d=1$ and 
\begin{equation*}
K_t(x,y) = \frac{\sin\pi(x-y)}{\pi(x-y)} \text{ for all } t
\end{equation*} 
with $\left| I_t \right| \rightarrow \infty$.  The original paper contains a comment, due to Widom, that the result holds for more general kernels.  
\end{remark}

\begin{remark}
We will use the result that a locally trace class Hermitian operator $K$ defines a determinantal random point field if and only if $0 \leq K \leq 1$ (see \cite{pe} or \cite{so1}).  
\end{remark}

\begin{proof}[Proof of Theorem \ref{general_costin_lebowitz}]
We first start by introducing some notation.  For a random variable $X$, let $C_l(X)$ denote the $l$th cumulant of $X$ and $F_l(X)$ denote the $l$th factorial moment of $X$.  By definition,
\begin{align*}
	&\sum_{k=1}^\infty \frac{ (iz)^k } {k!} C_k(X) = \log \E\left[ e^{izX} \right], \\
	&F_l(X) = \E\left[X(X-1) \cdots (X-l+1) \right].
\end{align*}
By writing the characteristic function of $X$ in power series form and expressing moments in terms of factorial moments, we obtain the following relation
\begin{equation} \label{fact_moments_cumulants_eq}
	\sum_{k=0}^\infty \frac{ (e^{iz}-1)^k }{ k! } F_k(X) = \exp \left( \sum_{k=1}^\infty \frac{ (iz)^k } {k!} C_k(X) \right).
\end{equation}

In order to prove the theorem and show convergence in distribution, we will show that the cummulants of
\begin{equation*}
	\xi_t = \frac{ \nu_t - \E_t[\nu_t] } { \sqrt{ \var_t (\nu_t) } }
\end{equation*}
converge to the cumulants of the standard normal.  Since the first and second cumulants of $\xi_t$ are $0$ and $1$, respectivally, it is enough to show that the remaining cumulants vanish in the limit as $t \rightarrow \infty$.  In particular, we will show that $C_l(\nu_t) = O(C_2(\nu_t))$ for $l > 2$.  Since $C_2(\nu_t) = \tr(A_t - A_t^2) \rightarrow \infty$ as $t \rightarrow \infty$ by assumption, we would then have that
\begin{equation*}
	C_l(\xi_t) = \frac{ C_l(\nu_t) }{ (C_2(\nu_t))^{l/2} } \longrightarrow 0 \text{ for } l > 2
\end{equation*}
as $t \rightarrow \infty$.  

Thus, we have only to show that $C_l(\nu_t) = O(C_2(\nu_t))$ for $l > 2$.  In order to do so, we introduce the cluster functions $r_{t,k}$, which are given by
\begin{equation*}
	r_{t,k}(x_1,\ldots,x_k) = \sum_{m=1}^k \sum_G (-1)^{m-1} (m-1)! \prod_{j=1}^m \rho_{t,{\left|G_j\right|}}(x_{G_j})
\end{equation*}
where $G$ is a partition of $\{1,2,\ldots,k\}$ into $m$ parts $G_1,\ldots,G_m$ and $x_{G_j}$ denotes the collection of $x_i$ with indices in $G_j$.  Let 
\begin{equation*}
	T_k(\nu_t) = \int_{I_t} \cdots \int_{I_t} r_k(x_1,\ldots,x_k) \ud x_1 \ldots \ud x_k.
\end{equation*}

For a determinantal random point process, we can write, 
\begin{align*}
	F_k(\nu_t) & = \int_{I_t} \cdots \int_{I_t} \rho_{t,k}(x_1,\ldots,x_k) \ud x_1 \ldots \ud x_k \\
		& = \int_{I_t} \cdots \int_{I_t} \det \left( K_t(x_i, x_j) \right)_{1 \leq i,j \leq k} \ud x_1 \ldots \ud x_k \\
		& = \sum_G \prod_{i=1}^m (-1)^{\left|G_i\right|} \int_{I_t} \ldots \int_{I_t} r_{\left|G_i\right|}(x_{G_i}) \ud x_{G_i} \\
		& = \sum_G \prod_{i=1}^m T_{\left|G_i\right|}(\nu_t) \\
		& = \sum_{k_1 + \cdots + k_m = k} \frac{k!}{k_1! \cdots k_m!} \frac{1}{m!} T_{k_1}(\nu_t) \cdots T_{k_m}(\nu_t)
\end{align*}
where $G$ is a partition of $\{1,2,\ldots,k\}$ into $m$ parts $G_1,\ldots,G_m$ and $k_i \geq 1$ for each $1 \leq i \leq m$.  Thus, we can write the generating function relation
\begin{equation*}
	\sum_{k=0}^\infty \frac{z^k}{k!} F_k(\nu_t) = \exp \left( \sum_{k=1}^\infty \frac{z^k}{k!} T_k(\nu_t) \right).
\end{equation*}
Using the relation between cumulants and factorial momements \eqref{fact_moments_cumulants_eq}, we obtain
\begin{equation} \label{cumulant_gen_fun_eq}
	\sum_{k=1}^\infty \frac{(iz)^k}{k!} C_k(\nu_t) = \sum_{k=1}^\infty \frac{(e^{iz}-1)^k}{k!} T_k(\nu_t).
\end{equation}
Finally, for determinantal random point fields
\begin{equation*}
	T_l(\nu_t) = (-1)^l (l-1)! \tr(A_t)^l
\end{equation*}
and hence by equating coefficients in \eqref{cumulant_gen_fun_eq} we have that
\begin{equation} \label{cumulant_rec_rel_eq}
	C_l(\nu_t) = (-1)^l (l-1)! \tr(A_t - A_t^l) + \sum_{s=2}^{l-1} \alpha_{s,l} C_s(\nu_t)
\end{equation}
where $\alpha_{s,l}$, $2 \leq s \leq l-1$ are some combinatorial coefficients (irrelevant for our purposes).  

We follow Soshnikov's example from \cite{so2} and bound the trace term
\begin{align*}
	0 \leq \tr(A_t - A_t^l) & = \sum_{j=1}^{l-1} \tr(A_t^j - A_t^{j+1}) \\
		& \leq \sum_{j=1}^{l-1} \|A_t^{j-1} \| \cdot \tr(A_t - A_t^2) \leq (l-1) C_2(\nu_t).
\end{align*}
Therefore, by an induction argument and equation \eqref{cumulant_rec_rel_eq}, we conclude that $C_l(\nu_t) = O(C_2(\nu_t))$ for $l > 2$ and hence the result follows.  
\end{proof}

\begin{remark}
The proof contained in \cite{pe} gives a much better probabilistic explaination of the result than the proof presented here.  In short, it states that $\nu_t$ has the same distribution as the sum of independent Bernoulli random variables.  Thus, \eqref{general_clt_eq} follows immediately from the Lindeberg-Feller Central Limit Theorem for triangular arrays (see \cite{gut}).  
\end{remark}

We now prove a multidimensional version of Theorem \ref{general_costin_lebowitz}.

\begin{theorem}[Soshnikov] \label{general_costin_lebowitz_multi}
Let $K_t$ be a family of locally trace class Hermitian operators associated with determinantal random point fields $\{\mathcal{P}_t\}_{t \geq 0}$ and let $\{I_t^{(1)},\ldots,I_t^{(s)} \}_{t \geq 0}$ be a famlily of Borel subsets of $\R^d$, disjoint for any fixed $t$, with compact closure.  Suppose 
\begin{align*}
	\var\left( \#\left( I_t^{(j)}\right) \right) &= \sigma_j^2 a_t (1+o(1)) \qquad 1 \leq j \leq s, \\
	\cov\left( \#\left( I_t^{(i)}\right), \#\left( I_t^{(j)}\right) \right) &= \gamma_{i,j} a_t (1+o(1)) \qquad i \neq j
\end{align*}
for some positive sequence of real numbers $\{a_t\}_{t \geq 0}$ such that $a_t \rightarrow \infty$ as $t \rightarrow \infty$.  Then the random vector
\begin{equation*} 
	\left( \frac{ \#\left( I_t^{(1)}\right) - \E\left[\#\left( I_t^{(1)}\right)\right] }{ \sqrt{a_t} }, \ldots, \frac{ \#\left( I_t^{(s)}\right) - \E\left[\#\left( I_t^{(s)}\right)\right] }{ \sqrt{a_t} } \right)
\end{equation*}
converges in distribution to the $s$-dimensional normal distribution $N(0,\Lambda)$ where $\Lambda_{i,j} = \gamma_{i,j}$ for $i \neq j$ and $\Lambda_{i,i} = \sigma^2_i$ for $1 \leq i \leq s$.  
\end{theorem}

\begin{remark}
The multidimensional case was proven by Soshnikov in \cite{so2} in the context of the Airy, Bessel, and sine kernels.  However, the proof given by Soshnikov is more general and applies to general determinantal random point fields.  
\end{remark}

In order to prove this result, we will need the following lemma, \cite{rs}.

\begin{lemma} \label{trace_bound}
If $A$ and $B$ are bounded operators on a separable Hilbert space $\mathcal{H}$ and $B \geq 0$ is trace class, then
\begin{equation*}
	\left| \tr\left(AB\right) \right| \leq \|A \| \tr \left(B \right).
\end{equation*}
\end{lemma}

Also, we will need that in the space of Hilbert-Schmidt operators on a separable Hilbert space $\mathcal{H}$, $(A,B) = \tr(A^*B)$ defines an inner product, \cite{rs}.  Thus, by the Cauchy-Schwarz inequality we have that
\begin{equation} \label{trace_cauchy_schwarz}
	\left| \tr\left(AB \right) \right| \leq \sqrt{ \tr\left(A^*A \right) }  \sqrt{ \tr\left(B^*B \right) }.
\end{equation}

\begin{proof}[Proof of Theorem \ref{general_costin_lebowitz_multi}]
We begin by introducing some notation.  Let $k=(k_1,\ldots,k_s)$ be a multi-index.  We define $\left|k\right| = k_1 + \cdots + k_s$ and $k! = k_1! \cdots k_s!$.  Let $z=(z_1,\ldots,z_s)$ be an $s$-vector.  We will use the following notation
\begin{align*}
	z^k &= z_1^{k_1} \cdots z_s^{k_s}, \\
	\left(e^{iz}-1\right)^k &= \left(e^{iz_1}-1\right)^{k_1} \cdots \left(e^{iz_s}-1\right)^{k_s}.
\end{align*}
For a multi-index $l=(l_1,\ldots,l_s)$ let $C_l$ denote the $l$th joint cumulant and $F_l$ denote the $l$th joint factorial moment of the random variables 
\begin{equation*}
	\#\left( I_t^{(1)}\right), \ldots, \#\left( I_t^{(s)}\right).
\end{equation*}
That is,
\begin{align*}
	&\sum_{k > 0} \frac{(iz)^k}{k!} C_k = \log \E \left[ e^{iz \cdot X_t} \right], \\
	&F_l = \E\left[ \prod_{j=1}^s \#\left( I_t^{(j)}\right) \left( \#\left( I_t^{(j)}\right) - 1\right) \cdots \left(\#\left( I_t^{(j)}\right) - l_j + 1 \right) \right]
\end{align*}
where $X_t$ is the $s$-dimensional random vector whose $j$th component is given by $\#\left( I_t^{(j)}\right)$.  Just as in the one-dimensional case, we have a relation between the joint factorial moments and the joint cumulants,
\begin{equation} \label{fact_moments_cumulants_multi_eq}
	\sum_{k\geq 0} \frac{ (e^{iz}-1)^k }{ k! } F_k = \exp \left( \sum_{k>0} \frac{ (iz)^k } {k!} C_k \right).
\end{equation}

The idea of the proof is to show that the joint cumulants $C_l$ vanish in the limit when $t \rightarrow \infty$ for $\left|l\right| > 2$.  In particular, we will show that $C_l = O(a_t)$ for all $\left|l\right| > 2$.  

We use the cluster functions $r_{t,n}$, which are given by
\begin{equation*}
	r_{t,n}(x_1,\ldots,x_n) = \sum_{m=1}^n \sum_G (-1)^{m-1} (m-1)! \prod_{j=1}^m \rho_{t,{\left|G_j\right|}}(x_{G_j})
\end{equation*}
where $G$ is a partition of $\{1,2,\ldots,n\}$ into $m$ parts $G_1,\ldots,G_m$ and $x_{G_j}$ denotes the collection of $x_i$ with indices in $G_j$.  Let $T_k$ to be the integral of $r_{t,{\left|k\right|}}$ over the region
\begin{equation*}
	{I_t^{(1)}}^{k_1} \times \cdots \times {I_t^{(s)}}^{k_s}.
\end{equation*}
Following a similar argument as in the proof of Theorem \ref{general_costin_lebowitz}, we obtain a multi-dimensional analogue of equation \eqref{cumulant_gen_fun_eq},
\begin{equation*}
	\sum_{k\geq 0}\frac{z^k}{k!} F_k = \exp \left( \sum_{k>0} \frac{z^k}{k!} T_k \right)
\end{equation*}
and hence by \eqref{fact_moments_cumulants_multi_eq}, we can write
\begin{equation} \label{cumulant_gen_fun_multi_eq}
	\sum_{k>0} \frac{(iz)^k}{k!} C_k = \sum_{k>0} \frac{(e^{iz}-1)^k}{k!} T_k.
\end{equation}
We can now obtain a recursive relation for $C_l$ in terms of $T_l$ as we did in the one-dimensional case.  If only one index of $l$ is non-zero, we are in the one-dimensional case and obtain \eqref{cumulant_rec_rel_eq}.  Since we dealt with this case in Theorem \ref{general_costin_lebowitz}, we will assume that $l$ contains at least two non-zero indices.  In this case, we equate coefficients from equation \eqref{cumulant_gen_fun_multi_eq} and obtain
\begin{equation} \label{cumulant_rec_rel_multi_eq}
	C_l = T_l + \sum_{2 \leq \left|k\right| < \left|l\right|} \alpha_{k,l} C_k
\end{equation}
where $\alpha_{k,l}$, $2 \leq \left|k\right| < \left|l\right|$ are some combinatorial coefficients (irrelevant for our purposes).  For a determinantal random point field, $T_k$ can be expressed as a linear combination of traces of the form
\begin{equation} \label{trace_form_eq}
	\tr \left( \chi_{I_t^{(j_1)}} \cdot K_t \cdot \chi_{I_t^{(j_1)}} \cdot K_t \cdot \chi_{I_t^{(j_2)}} \cdots K_t \cdot \chi_{I_t^{(j_m)}} \cdot K_t \cdot \chi_{I_t^{(j_1)}} \right)
\end{equation}
such that if $k_i$ is nonzero then at least one of the indicators in each term in the linear combination is the indicator of $I_t^{(i)}$.  Therefore, using the bounds in Lemma \ref{trace_bound} and \eqref{trace_cauchy_schwarz}, we can bound the trace in \eqref{trace_form_eq} by terms of the form
\begin{equation*}
	\tr \left( \chi_{I_t^{(j)}} \cdot K_t  \cdot \chi_{I_t^{(i)}} \cdot K_t  \cdot \chi_{I_t^{(j)}} \right) = O(a_t)
\end{equation*}
where $i \neq j$ or terms of the form
\begin{equation*}
	\sqrt{ \tr \left( \chi_{I_t^{(j)}} \cdot K_t  \cdot \chi_{I_t^{(i)}} \cdot K_t  \cdot \chi_{I_t^{(j)}} \right) } \sqrt{ \tr \left( \chi_{I_t^{(\alpha)}} \cdot K_t  \cdot \chi_{I_t^{(\beta)}} \cdot K_t  \cdot \chi_{I_t^{(\alpha)}} \right) } = O(a_t)
\end{equation*}
where $i \neq j$ and $\alpha \neq \beta$.  Hence the result follows by an induction argument on \eqref{cumulant_rec_rel_multi_eq}.  
\end{proof}

\section*{Acknowledgment}
I would like to thank my advisor Alexander Soshnikov for the discussions that initiated this work and for the invaluable guidance and support that helped complete it.  I also thank Van Vu for the very useful comments.  

This material is based upon work supported by the National Science Foundation under Grant \#DMS-0636297.

\end{document}